\documentclass[a4paper,11pt,reqno]{amsart}
\usepackage{epsf,epsfig}
\usepackage{graphics,color}
\usepackage{amsmath}
\usepackage{amssymb}
\usepackage{cite}
\usepackage{verbatim}
\usepackage{float}
\usepackage{graphicx}
\usepackage{amsthm}
\usepackage{subfig}

\usepackage[colorlinks=true, hyperindex, pageanchor, raiselinks, bookmarks,
	 menucolor=orange!40!black,
        linkcolor=blue!40!black,
        citecolor=green!40!black,
        filecolor=magenta!40!black, 
        urlcolor=blue!40!black,
        ]{hyperref}
\hypersetup{colorlinks=true}

\usepackage[pdftex,dvipsnames]{xcolor}
\usepackage{xargs}
\usepackage[colorinlistoftodos,prependcaption]{todonotes}
\newcommandx{\unsure}[2][1=]{\todo[linecolor=red,backgroundcolor=red!25,bordercolor=red,#1]{#2}}
\newcommandx{\change}[2][1=]{\todo[linecolor=blue,backgroundcolor=blue!25,bordercolor=blue,#1]{#2}}
\newcommandx{\improvement}[2][1=]{\todo[linecolor=Plum,backgroundcolor=Plum!25,bordercolor=Plum,#1]{#2}}
\newcommandx{\thiswillnotshow}[2][1=]{\todo[disable,#1]{#2}}

\renewcommand{\d}{\mathrm{d}}
\newcommand{\D}{\mathcal{D}}
\newcommand{\J}{\mathcal J}

\newcommand{\RR}{{\mathbb R}}
\newcommand{\NN}{\mathbb{N}}

\newcommand{\haz}{\widehat}

\newcommand{\We}{W_{\rm e}}

\newcommand{\Fe}{F_{\rm e}}
\newcommand{\ye}{y_{\rm e}}
\newcommand{\hye}{{\haz y}_{\rm e}}
\newcommand{\hyp}{{\haz y}_{\rm p}}
\newcommand{\hyen}{{\haz y}_{\rm e}^n}
\newcommand{\qe}{{q_{\rm e}}}
\newcommand{\yp}{y_{\rm p}}
\newcommand{\qp}{{q_{\rm p}}}

\newcommand{\ypo}{y_{{\rm p}0}}
\newcommand{\yeo}{y_{{\rm e}0}}
\newcommand{\ypi}{y_{{\rm p}i}}
\newcommand{\yei}{y_{{\rm e}i}}
\newcommand{\ypj}{y_{{\rm p}j}}
\newcommand{\yej}{y_{{\rm e}j}}
\newcommand{\ypii}{y_{{\rm p}(i-1)}}
\newcommand{\yeii}{y_{{\rm e}(i-1)}}
\newcommand{\ypk}{y_{{\rm p}k}}
\newcommand{\yek}{y_{{\rm e}k}}
\newcommand{\Wp}{W_{\rm p}}

\newcommand{\Fp}{F_{\rm p}}
\newcommand{\Fpo}{F_{\rm p0}}
\newcommand{\Fpoo}{F_{\rm p2}}
\newcommand{\Fpi}{F_{{\rm p}i}}
\newcommand{\Fpj}{F_{\rm p1}}
\newcommand{\GL}{{\rm GL}(d)}

\newcommand{\SL}{{\rm SL}(d)}
\newcommand{\SO}{{\rm SO}(d)}
\newcommand{\id}{{\rm id}}
\newcommand{\cof}{{\rm  cof\,}}
\renewcommand{\det}{{\rm  det\,}}
\newcommand{\curl}{{\rm curl\,}}
\newcommand{\Diss}{\textrm{Diss}_\D}
\DeclareMathOperator{\argmin}{arg\,min}
\newcommand{\E}{\mathcal{E}}
\newcommand{\wto}{\rightharpoonup}
\newcommand{\wstar}{\overset{\ast}{\rightharpoonup}}
\renewcommand{\d}{{\rm d}}
\def\argmin{\,{\rm argmin \ }}

\def\W{\mathcal W}
\def\Q{\mathcal Q}
\def\eps{\varepsilon}
\newcommand{\dist}{\mathrm{dist}}

\newcommand{\Om}{\Omega}

\definecolor{myred}{rgb}{0.8, 0, 0.1}

\numberwithin{equation}{section}
\mathchardef\emptyset="001F

\newtheorem{theorem}{Theorem}[section]

\newtheorem{lemma}[theorem]{Lemma}

\theoremstyle{definition}
\newtheorem{remark}[theorem]{Remark}
\newtheorem{definition}[theorem]{Definition}

\title[Quasistatic evolution for dislocation-free finite
plasticity]{Quasistatic evolution for \\ dislocation-free finite plasticity}

\author{Martin Kru\v z\'ik} 
\address[Martin Kru\v z\'ik]{Institute of Information Theory and Automation, Czech Academy of Sciences, Pod vod\'arenskou ve\v z\'i 4, 182 08 Prague, Czechia and Faculty of Civil Engineering, Czech Technical University, Th\'{a}kurova 7, 166 29 Prague, Czechia}
\email{kruzik@utia.cas.cz}
\urladdr{http://staff.utia.cas.cz/kruzik/}

\author{David Melching}
\address[David Melching]{Faculty of Mathematics, University of Vienna, 
Oskar-Morgenstern-Platz 1, 1090 Wien, Austria}
\email{david.melching@univie.ac.at}
\urladdr{http://www.mat.univie.ac.at/$\sim$melching}

\author{Ulisse Stefanelli} 
\address[Ulisse Stefanelli]{Faculty of Mathematics, University of Vienna, 
Oskar-Morgenstern-Platz 1, 1090 Wien, Austria  and  Istituto di Matematica
Applicata e Tecnologie Informatiche \textit{{E. Magenes}}, v. Ferrata 1, 27100
Pavia, Italy.}
\email{ulisse.stefanelli@univie.ac.at}
\urladdr{http://www.mat.univie.ac.at/$\sim$stefanelli}
\date{\today}


\begin{document}

\begin{abstract}
We investigate quasistatic evolution in finite
plasticity under the assumption that the plastic strain is compatible. This
assumption is well-suited to describe the special case of dislocation-free
plasticity and entails that  the plastic strain is the gradient of a plastic
deformation map. The total deformation can be then seen as the
composition of a plastic and an elastic deformation. This opens
the way to an existence theory for the quasistatic evolution problem
featuring both Lagrangian and Eulerian variables. A remarkable trait
of the result is that it does not require second-order
gradients. 
\end{abstract}

\maketitle

\section{Introduction}

The elastoplastic behavior of a crystalline solid under
the action of external loads  results from a
combination of reversible elastic and irreversible plastic
effects \cite{maugin}. The state of the body 
is specified in terms of its deformation $y:\Omega \to \RR^3 $ from a reference
configuration $\Omega \subset \RR^3$. Elastic and plastic effect are
classically assumed to combine via the {\it  Kr\"{o}ner-Lee-Liu  multiplicative decomposition} of the  total  strain $\nabla y = F_{\rm e}F_{\rm p}$  \cite{Kroener,Lee-Liu67,Lee69}. Here, the elastic strain $F_{\rm e}$ describes
the elastic response of the medium, whereas the plastic strain $F_{\rm p}$
records the accumulation of plastic distortion \cite{Gurtin10}. In
metals, it is usually assumed that plastic effects induce no volume change,
namely $\det F_{\rm p}=1$ \cite{SH98}.

Elastoplastic evolution results from the competition of elastic-energy
storage and plastic-dissipation mechanisms. As such, a common and successful
approach to the description of elastoplasticity of crystalline materials is via variational methods 
\cite{ortiz}. 
The energy of the specimen is often assumed to be of the form
\begin{equation}
\int_\Omega \We(\nabla y F_{\rm p}^{-1}) \, \d x + \int_\Omega
\Wp(F_{\rm p})\, \d x,
\label{E}
\end{equation}
where $\We$ is the elastic-energy density, a function of the elastic
strain  $F_{\rm e}=\nabla y F_{\rm p}^{-1}$, and $\Wp$ is a
hardening-energy density. In the {\it incremental} setting of the
elastoplastic evolution problem, given external loads and boundary conditions, 
one minimizes the energy, augmented by a dissipation term $\mathcal
D(\Fpo, \Fp)$ \cite{Mielke03b}. The latter measures the distance of the actual
plastic strain $F_{\rm p}$ from the previous $F_{\rm p0}$. This
inspires different solution notions on the time-continuous, {\it quasistatic} evolution level \cite{Mielke-Roubicek}.

In view of the mathematical treatment of
finite plasticity, one is hence confronted with the necessity of controlling the product $\nabla y F_{\rm p}^{-1}$. This is
indeed a critical point, for weak topologies are not sufficient in
order to identify this product within a corresponding limit passage. Such observation has sparked the interest for so-called {\it second-order} theories, where a term featuring the gradient $\nabla F_{\rm p}$ is included in the
energy. This gradient term models nonlocal effects caused by
short-range interactions among dislocations \cite{dillon,gurtin,gurtin-anand}. From a mathematical standpoint, the presence of the gradient $\nabla F_{\rm p}$ in the energy contributes strong compactness for $ F_{\rm p}$,
which then allows to pass to the
limit in the product $\nabla y F_{\rm
p}^{-1}$.

To date, multidimensional existence results for incremental and quasistatic evolutions are just a few and all hinge on second-order
theories \cite{Davoli1,cplas_part2,souza,mainik-mielke2,Mielke-Mueller,mielke-roubicek16}. However, in finite plasticity these second-order gradient theories are still debated from the modeling
standpoint. In particular, it is not clear which function of the
gradient should be used. We refer to \cite{kratochvil,bako,zaiser} for
attempts  to derive it from statistical physics, revealing the complexity
of this issue. A related  approach to nonlocal models in damage
and plasticity  was undertaken in \cite{bazant}, see also \cite{fleck-hutch1,fleck-hutch2,fremond,maugin}.

Our aim is to investigate existence for
quasistatic evolutions not relying on second-order theories, namely
in absence of a regularizing gradient term $\nabla F_{\rm p}$. This
follows the analysis of \cite{Stefanelli18}, where the same issue was
considered at the incremental level. The price to pay for allowing
such an existence result is that of restricting the analysis to the
case of  {\it compatible} plastic strains $ F_{\rm p}$, namely to impose $\curl F_{\rm p}=0$. This case corresponds to {\it dislocation-free} elastoplastic
evolution. Albeit not generic, such situation  still includes plastic slips 
\cite{Conti-Reina} and may actually
occur in ductile metals \cite{Kiritani,Matsukawa}. This is
particularly relevant in case of small bodies. Indeed, dislocation dynamics is strongly size-dependent \cite{Greer,Uchic} so that very small dislocation-free bodies may
plasticize without nucleating dislocations. 

In case of compatibility, one can identify the plastic strain  $F_{\rm
  p}$ with a gradient of a {\it plastic deformation} $\yp:\Omega \to
\yp(\Omega)\subset \RR^3$, mapping indeed the reference configuration to the so-called
{\it intermediate} one. At the same time, this defines an {\it elastic
  deformation} $\ye: \yp(\Omega)\to \RR^3 $ from the intermediate to
the actual configuration such that the decomposition
\begin{equation}\label{compos}
y = \ye \circ \yp
\end{equation}
holds. The latter of course entails the multiplicative decomposition
$\nabla y = \nabla \ye \nabla \yp$ via the classical chain rule. On
the other hand, by assuming $\yp$ to be injective, it allows for
rewriting the energy in \eqref{E}, by a change of variables, as
\begin{equation}
\int_{\yp(\Omega)} \We(\nabla \ye) \, \d \xi + \int_\Omega
\Wp(\nabla \yp)\, \d x.\label{E2}
\end{equation}
This reformulation of the energy is particularly advantageous from the
mathematical viewpoint, for it does not feature the product term
$\nabla y F_{\rm p}^{-1}$ anymore. This in turn allows for an existence theory via
classical variational methods, even in absence of strong compactness for $F_{\rm
  p}=\nabla \yp$. Indeed, in two space dimensions, by assuming $F_{\rm
  p}=\nabla \yp$ one would even be able to directly identify the limit in $\nabla y
(\nabla \yp)^{-1} $ via the classical div-curl lemma as $\nabla y$ is curl-free and ${\rm div} \, (\nabla \yp)^{-T}= 0$ if $\det \nabla \yp=1$, see also \cite{Conti11}.

Arguing via reformulation \eqref{E2}
calls for the treatment of both {\it Lagrangian} and {\it Eulerian}
terms, respectively defined on the reference and on the intermediate
configuration, which itself depends on part of the solution. This kind
of mixed Lagrangian-Eulerian problems has to be traced back at least to
\cite{Fonseca-Parry,Dacorogna-Fonseca88} for the case of defective
crystals and has recently attracted
attention in connection with nematic elastomers
\cite{Barchiesi-DeSimone,Barchiesi}, magnetoelasticity
\cite{Barchiesi,rainy,Rybka,Sil18,Tomassetti}, solid-solid phase change
\cite{GKMS18,Sil11} and, as already mentioned,
incremental finite plasticity \cite{Stefanelli18}. A decomposition of type \eqref{compos} has recently also been used as a starting point to model dissolution-precipitation creep \cite{Hackl15}.

The main result of this paper is the existence of {\it incrementally approximable
quasistatic evolutions},  see Theorem
\ref{existence:quasistatic}.  This  
notion of solution features stability and energy balance on the
time-discrete level as well as  {\it semistability} relation with
respect to elastic deformations and an {\it energy inequality} in the
time-continuous limit.  A similar notion  has been considered in
\cite{DalMaso-Lazzaroni-2010} in the quasistatic setting and in
\cite{RoegerSchweizer17}  for viscoplasticity  and is weaker than the concept of {\it energetic solutions} \cite{Mielke-Roubicek}.  Still, it implies the validity of
the quasistatic equilibrium system as well as the dissipative character of the
evolution. 

The existence proof follows the classical time-discretization
strategy. Discrete-in-time solutions are found by solving incremental
problems on a given time partition and a quasistatic evolution is then
recovered as the fineness of the partitions tends to zero. In order to
check for the energy inequality, the
lower semicontinuity of the energy and dissipation functionals plays a
crucial role. This results from the {\it weak compactness} of the minors
of $\nabla \ye$ and $\nabla \yp$ (see \eqref{E2}) under the assumption
of {\it polyconvex} densities \cite{Ball76}. The passage to the limit
in the discrete semistability requires an ad hoc recovery-sequence
construction, which in turn hinges upon the possibility of extending
elastic deformations to a neighborhood of the intermediate
configuration. In order to be able to achieve this, intermediate
configurations are asked to have regular boundaries. More precisely, they are restricted to belong to a certain uniform subclass of Jones domains \cite{Jones}, see
Subsection \ref{domains}.

The mechanical model and its variational formulation are introduced in
Section \ref{sec:setting} and
the main result is stated in Subsection \ref{sec:statement}. The existence proof is then detailed in Section \ref{proofs}.

\section{Main result}\label{sec:setting}
This section brings us to the formulation of our main result,
Theorem \ref{existence:quasistatic}. We start by introducing our
 assumptions and basic framework in Subsections \ref{notation}-\ref{Diss} and end with our main statement in Subsection \ref{sec:statement}.

\subsection{Notation}\label{notation}
In what follows, we denote by $\RR^{d\times d}$ the Euclidean
space of $d\times d$ real matrices and by $\SL,\GL$, and $\SO$ its subspaces of matrices with unit determinant, invertible matrices, and proper rotations, respectively. Using $\mathcal L^d$ and $\mathcal H^{d-1}$ we refer to the $d$-dimensional Lebesgue measure and the
$(d-1)$-dimensional Hausdorff measure. 
The norm on a generic Banach space $E$ is denoted by $\|
\cdot \|_E$ and we use the standard notation for Sobolev and Lebesgue spaces. By default, we denote by $f_n \to f$ strong convergence, whereas $f_n \wto f$ means weak convergence.

\subsection{Deformations}\label{deformations}
Let $d \ge 2$ and $\Omega \subset \RR^d$ be a non-empty, open,
simply connected, bounded domain with Lipschitz boundary. The boundary is
essentially split into a Dirichlet part $\Gamma_D$ and a Neumann part
$\Gamma_N$, namely $\partial \Om = \overline{\Gamma}_D \cup
\overline{\Gamma}_N$ with $\Gamma_D$ and $\Gamma_N$ open in $\partial
\Om$ and $\Gamma_D \cap \Gamma_N = \varnothing$ where $\mathcal
H^{d-1}(\Gamma_D) >0$. We indicate by $y: \Om \to \RR^d$ the
deformation of the body $\Om$. 

The crucial assumption of our theory is that the deformation $y$ can be decomposed into elastic and
plastic deformations $\ye$ and $\yp$ as in \eqref{compos}. As
mentioned, this follows from the standard multiplicative decomposition
$\nabla y= F_{\rm e} F_{\rm p}$ in case $F_{\rm p}$ is
 curl-free. Indeed, if $F_{\rm p}=\nabla \yp$ for some plastic
deformation $\yp$ one can easily check
\cite{Stefanelli18} that $F_{\rm e}=\nabla \ye$ for some elastic
deformation $\ye$, so that the multiplicative decomposition
$\nabla y= \nabla \ye \nabla \yp$  follows by \eqref{compos} and the
classical chain rule.  We now detail our assumptions on $\yp$ and $\ye$.

\textit{Plastic deformations.} We assume that the plastic deformation fulfills 
$$\yp \in W^{1,\qp}(\Om;\RR^d) \quad \text{for some 
$\qp > d(d-1)$}$$  
and that it is locally volume preserving,
namely, $\det \nabla \yp = 1$ almost everywhere ({a.e.})  in $\Om$ \cite{Mandel72,SH98}. This
implies that $\yp$ is H\"older continuous with exponent $1-d/\qp$ and
almost everywhere differentiable \cite[Lemma 2.7]{Fonseca-Gangbo}.  From now on, when writing $\yp$ we always mean its continuous representative.
 
The map $\yp$ possesses the so-called Lusin's {\it $N$-property}, namely
$\mathcal L^d(E)=0 \Rightarrow \mathcal L^d(\yp(E))=0$ for all measurable $E \subset \RR^d$, as well as the corresponding {\it $N{^{-1}}$-property}, i.e.
$\mathcal L^d(E')=0 \Rightarrow \mathcal L^d(\yp^{-1}(E'))=0$ for all measurable $E' \subset \RR^d$, see \cite[p. 296]{Goldshtein90}. Moreover, $\yp$ is \textit{locally invertible almost everywhere} \cite[Thm. 3.1, Cor. 3.3]{Fonseca-Gangbo}. This means that for a.e. $x \in \Om$ there exists a ball $B \subset \RR^d$ centered at $\yp(x)$, an open neighborhood $U \subset \Om$ of $x$,  and a local inverse $\yp^{-1}:B \to U$ with $\yp^{-1} \in W^{1,\qp/(d-1)}(B;\RR^d)$ such that $\yp|_{U}$ and $\yp^{-1}$ are onto, $\yp ^{-1}\circ \yp = \id $
a.e.  in $U$,   $\yp\circ \yp ^{-1} = \id $ a.e.  in $B$, and $\nabla \yp^{-1} = (\nabla \yp)^{-1}\circ \yp^{-1}$ a.e.  in $B$.

In view of changing  from Lagrangian to Eulerian variables, we require $\yp$ to be {\it injective almost everywhere}, namely, that there exists a negligible set $N$ such that $\yp$ is injective on $\Omega \setminus N$. This property is implemented by imposing the classical \textit{Ciarlet-Ne\v cas condition} \cite{ciarlet-necas}
\begin{equation}\label{Ciarlet-Necas}
\mathcal{L}^d(\Om) = \int_{\Om} \det (\nabla \yp(x)) \,\d x \le \mathcal{L}^d(\yp(\Om)).
\end{equation}
In this setting, \eqref{Ciarlet-Necas} and injectivity almost everywhere are actually equivalent \cite[Prop. 15]{GP08}. Using injectivity almost everywhere, we  get  the change of variables formula
\begin{equation}
\label{change}
\int_E \phi(\yp(x)) \,\d x = \int_{\yp(E)} \phi(\xi) \,\d \xi
\end{equation}
for every measurable function $\phi:\Om \to \RR^d$ and all measurable
$E \subset \RR^d$, see  \cite[Lem. 2.4]{Fonseca-Gangbo}. Note
that, here and in the following, we use the shorthand $\d x$ for $\d
\mathcal L^d(x)$ when  integrating with respect to Lagrangian coordinates
$x\in \Omega$, and $\d \xi$ for $\d \mathcal L^d(\xi)$ in case of Eulerian
coordinates, namely for integration on the intermediate configuration
$\yp(\Omega)$. 
 
If $\yp \in W^{1,d}(\Om; \RR^d)$ with distortion $K:=|\nabla \yp|^d/\det \nabla \yp \in L^p(\Om;\RR)$ for $p > d-1$, then $\yp$ is either constant or open, \cite[Theorem 3.4]{HenclKoscela14}. Since in our setting $\qp > d(d-1)$ and $\det \nabla \yp =1$,  this integrability requirement is exactly fulfilled. Moreover, by the Ciarlet-Ne\v cas condition \eqref{Ciarlet-Necas}, $\yp$ cannot be constant, which shows that $\yp$ is  open and injective almost everywhere. This implies that $\yp$ is (globally) injective \cite[Lemma 3.3]{GKMS18}, and that $\yp$ is actually a {\it homeomorphism} having inverse $\yp^{-1}$ of regularity
\begin{equation*}
\yp^{-1} \in W^{1,\qp/(d-1)}(\yp(\Om);\RR^d).
\end{equation*}
Note that, if the plastic deformation $\yp$ at the boundary $\partial
\Omega$ was coinciding with that of a  homeomorphism  on $\overline \Omega$, given the integrability of the  distortion  one could resort to the invertibility theory by Ball
\cite{Ball88} to deduce that $\yp$ is actually a homeomorphism, even
without asking for the Ciarlet-Ne\v cas condition
\eqref{Ciarlet-Necas}.  In our case however, we cannot assume to be
able to prescribe $\yp(\partial \Omega)$, for $\yp$ is an internal
variable. In fact, since
the problem is formulated in terms of $\nabla \yp$ only, we
later ask for the normalization condition $\int_\Omega \yp(x)\d x=0$.

\textit{Elastic deformations.} Given a plastic deformation $\yp$, we assume the elastic deformation $\ye$, defined on the intermediate configuration $\yp(\Om)$, to satisfy
\begin{equation*}
\ye \in W^{1,\qe}(\yp(\Om);\RR^d)\quad \text{for some $\qe > d$.}
\end{equation*}
By using the local invertibility of $\yp$, one checks that the chain rule
\begin{equation}
\label{chain}
\nabla y(x) = \nabla \ye (\yp(x))\, \nabla \yp(x)
\end{equation}
holds for almost every $x \in \Om$, see \cite{Stefanelli18} for details. We can use the change of variables formula \eqref{change} together with the chain rule \eqref{chain} and H\"older's inequality to  estimate 
\begin{equation}
\label{chain:est}
\int_\Omega |\nabla y(x)|^{q}\, \d x \le \left(\int_{\yp(\Omega)} | \nabla \ye(\xi)
|^\qe\, \d \xi \right)^{q/\qe} \left(\int_\Omega |\nabla \yp(x)|^\qp \, \d x
\right)^{q/\qp}, 
\end{equation}
where $q$ is defined by
\begin{equation*}
\frac{1}{q} = \frac{1}{\qe} + \frac{1}{\qp}.
\end{equation*}



\subsection{Domains}\label{domains}
In order to carry on our existence proof, some regularity for the intermediate configurations $\yp(\Om)$  is  needed. Our goal is to find conditions under which $\yp(\Om)$ is a {\it Sobolev extension domain}. These are open subsets of $\RR^d$ allowing the extension of Sobolev functions to the whole space. More precisely, $\omega \subset \RR^d$ is a {\it $W^{1,p}$-extension domain}, if and only if one can define a bounded linear operator 
\begin{equation*}
E : W^{1,p}(\omega;\RR^d) \to W^{1,p}(\RR^d;\RR^d)
\end{equation*}
such that $$Eu = u \quad \text{ in } \omega$$ 
for every $u \in W^{1,p}(\omega;\RR^d)$. 
 Additionally, we need to ensure that the class of intermediate domains is closed under Hausdorff convergence of sets, in order to guarantee that the state space is closed. The Hausdorff distance of two non-empty, compact subsets $X,Y$ of $\RR^d$ is defined as 
\begin{equation*}
d_{\rm H}(X,Y) := \inf \{\nu \ge 0: X \subset B_{\nu}(Y), Y \subset B_{\nu}(X)\},
\end{equation*}
where $B_\nu(X):= \{z \in \RR^d: \text{ there exists } x \in X \text{ such that } |x-z| < \nu \} = X + B_\nu(0)$ is an $\nu$-fattening of the set $X$. It is easy to see that 
\begin{equation*}
d_{\rm H}(X,Y) = \max \bigg\{ \sup_{x \in X} \ \dist(x,Y), \;\; \sup_{y \in Y} \ \dist(y,X)\bigg\}, 
\end{equation*}
where $\dist(x,Y):= \inf_{y\in Y} |x-y|$.
We remark that, if 
$\yp^n \wto \yp$ in $W^{1,\qp}(\Om; \RR^d),$
then $\yp^n$ converges uniformly to $\yp$ on $\overline{\Om}$ by the compact Sobolev embedding $W^{1,\qp}(\Om;\RR^d) \subset\subset C^0(\overline{\Om};\RR^d)$. This implies Hausdorff convergence of
the intermediate configurations,  namely
\begin{equation}\label{Hausdorff-conv}
d_{\rm H}(\overline{\yp^n({\Om})},\overline{\yp({\Om})}) \to 0, \quad d_{\rm H}(\partial \yp^n(\Om),\partial \yp(\Om)) \to 0
\end{equation}
as $n$ tends to $\infty$,  see  Lemma \ref{closedness:A} below.

It is well-known that Lipschitz domains are $W^{1,p}$-extension domains for every $1 \le p \le \infty$ \cite{Calderon,Stein}. However, the class of Lipschitz domains is not closed under Hausdorff convergence. We hence focus here on a larger class of domains, the {\it $(\eps,\delta)$-domains} introduced by Jones in \cite{Jones} and defined below. These possess the extension property \cite[Theorem 1]{Jones} and include Lipschitz domains. By restricting to the subclass of uniform $(\eps,\delta)$-domains, we  obtain uniformly bounded extension operators as well as closedness with respect to Hausdorff convergence.

\begin{definition}[$\J_{\eps,\delta}$ domains]
  We say that a bounded, open set $\omega \subset \RR^d$ is an
  $(\eps,\delta)$-domain, denoted $\omega \in \J_{\eps,\delta}$, if
  for every $x,y \in \omega$ with $|x-y| < \delta$ there exists a
  Lipschitz curve $\gamma \in W^{1,\infty}([0,1];\omega)$ with
  $\gamma(0)=x$, and $ \gamma(1)=y$ satisfying the following two conditions:
  \begin{equation}\label{Jones:cond1}
    \ell(\gamma) := \int_0^1 |\dot \gamma(s)| \ \d s \le \frac{1}{\eps} |x-y|
  \end{equation} 
  and
  \begin{equation}\label{Jones:cond2}
    \text{dist}(\gamma(t),\partial \omega) \ge \eps \frac{|x-\gamma(t)||\gamma(t)-y|}{|x-y|} \quad \forall t \in [0,1].
  \end{equation}
\end{definition}

One can immediately see that these classes of domains are nicely ordered in the sense that if $\omega$ is an $(\eps',\delta')$-domain for some $\eps' \ge \eps$ and $\delta' \ge \delta$, then $\omega$ is also an $(\eps,\delta)$-domain. More precisely,
\begin{equation*}
\J_{\eps,\delta} = \bigcup_{\eps' \ge \eps, \delta'\ge \delta} \J_{\eps',\delta'}.
\end{equation*}

\subsection{States}
\label{states}
 Let $\eps, \delta >0$. We define the set of admissible states as
\begin{align*}
\Q := \Bigg\{&(\ye,\yp) \in W^{1,\qe}(\yp(\Om);\RR^d) \times W^{1,\qp}(\Om;\RR^d): \\
&\yp(\Om) \in \J_{\eps,\delta}, \quad
\int_{\Om} \yp \,\d x = 0, \quad \det \nabla \yp = 1 \text{ a.e. in }\Om, \quad \mathcal{L}^d(\Om) \le \mathcal{L}^d(\yp(\Om)) \Bigg\}.
\end{align*}
The state space $\Q$ is equipped with the weak topology of $W^{1,\qe}_{\rm loc}(\yp(\Om);\RR^d) \times W^{1,\qp}(\Om;\RR^d)$. More precisely, we  write that {\it$(\ye^n,\yp^n)_{n \in \NN} \subset \Q$ converges to $(\ye,\yp)$ in $\Q$}, if 
\begin{align*}
&\yp^n \wto \yp \text{ in } W^{1,\qp}(\Om;\RR^d), \\
&\ye^n \wto \ye \text{ in } W^{1,\qe}(K;\RR^d) \text{ for every }K \subset\subset \yp(\Om).
\end{align*}
Note that, since $W^{1,\qp}(\Om;\RR^d) \subset \subset C^0(\overline{\Om};\RR^d)$, for every $K \subset\subset \yp(\Om)$ there exists $n_K \in \NN$ such that $K \subset \yp^n(\Om)$ for all $n\ge n_K$. 
 In Section \ref{closedness} below, we prove (sequential) closedness
 of $\Q$ under this convergence.  The constraint $\yp(\Omega) \in
 {\mathcal J}_{\eps,\delta}$ is global in nature and is expected to be not restrictive in most
 practical cases. 

\subsection{Energy}\label{energy}
The stored energy corresponding to the state $(\ye,\yp) \in \Q$
consists of  three  parts: an elastic energy, which is defined on the
intermediate configuration $\yp(\Omega)$ and depends on the elastic strain $\nabla \ye$,  a kinematic hardening energy, depending solely on the plastic strain $\nabla \yp$, and a soft elastic boundary condition defined on the Dirichlet boundary $\Gamma_D$. More precisely, the stored energy of the system reads
\begin{equation*}
\W(\ye,\yp) = \int_{\yp(\Om)} \We(\nabla \ye(\xi)) \,\d \xi + \int_\Om \Wp(\nabla \yp(x)) \,\d x + \int_{\Gamma_D} |\ye(\yp(x)) - x| \, \d \mathcal H^{d-1}(x).
\end{equation*}
The system is driven by a time-dependent body force $f: [0,T] \times \Om \to \RR^d$ and a boundary traction $g: [0,T] \times \Gamma_N \to \RR^d$ (provided $\Gamma_N \neq \varnothing$) which result in an  external  loading $\ell$, and loading energy defined as
\begin{equation*}
\langle \ell(t), y\rangle = \int_\Om f(t,x) \cdot y(x) \, \d x + \int_{\Gamma_N} g(t,x) \cdot y(x) \, \d \mathcal H^{d-1}(x).
\end{equation*}
The total energy of the system is then given by
\begin{equation*}
\E(t,\ye,\yp) = \W(\ye,\yp) - \langle \ell(t), \ye \circ \yp \rangle.
\end{equation*}

We assume the elastic energy to have {\it$\qe$-growth} and the plastic energy density to be {\it coercive}, i.e.

\begin{subequations}
	\begin{align}
	\label{growth:elastic}
	c|\Fe|^\qe - \frac{1}{c} &\le \We(\Fe) \le \frac{1}{c}(1+|\Fe|^\qe),\\
	\label{coercive:plastic}
	c|\Fp|^\qp - \frac{1}{c}&\le \Wp(\Fp)
	\end{align}
\end{subequations}

for some constant $c>0$ and every $\Fe\in \RR^{d\times d}, \Fp \in \SL$. 
This is combined with the structural assumption of {\it polyconvexity}, namely
\begin{subequations}
	\label{polyconvexity}
	\begin{align}
	\label{polyconvex:1}
	\We(\Fe) &= \haz W_{\rm e} (\Fe, \cof \Fe, \det \Fe),\\
	\label{polyconvex:2}
	\Wp(\Fp) &= \haz W_{\rm p} (\Fp, \cof \Fp)
	\end{align}
\end{subequations}
where $\haz \We:\RR^{d\times d} \times \RR^{d\times d} \times \RR \to \RR$ and $\haz W_{\rm p}:\RR^{d\times d} \times \RR^{d\times d} \to\RR$ are convex. We remark that the notation corresponds to space dimension $d=3$, as the minors of a matrix are then given by determinant, cofactor, and the matrix itself. For $d=2$ the dependence on the cofactor matrix could be dropped and in dimensions $d>3$ the definition of polyconvexity could be generalized by including further minors.
Although not directly needed for the analysis, we may assume the energy to be frame-indifferent. This corresponds to asking the elastic energy density to satisfy the assumption  $\We(R\Fe) = \We(\Fe)$ for all $R\in \SO, \Fe\in \RR^{d\times d}$.
We further assume
\begin{align*}
f\in W^{1,1}(0,T;L^{(q^*)'}(\Om;\RR^d)), \quad g \in W^{1,1}(0,T;L^{(q^\#)'}(\Gamma_N;\RR^d)),
\end{align*}
where $q^*$ and $q^\#$ denote the {\it Sobolev exponent}
and the {\it trace exponent}, respectively,
and prime stands for conjugation \cite{Roubook}. Let us remark that  the assumptions on $\qe$ and $\qp$ contribute the following lower bounds on the mentioned exponents: 
\begin{equation*}
q > d-1,\qquad q^* > d(d-1), \qquad q^\# > (d-1)^2.
\end{equation*}
These assumptions ensure  that the loading is absolutely continuous in time, namely,
\begin{equation}\label{load:reg}
\ell \in W^{1,1}(0,T;(W^{1,q}(\Om;\RR^d))^*)
\end{equation}
where $*$ denotes the dual space.

\begin{remark}
	{\it Locking materials} may also be considered. These materials are characterized by a tolerance $M>0$ and internal energy defined as above if $\|\nabla \yp\|_{L^\infty(\Om;\RR^d)} \le M$, and $\W(\ye,\yp) =  \infty$ otherwise.
	This would force the plastic deformations to be (uniformly) Lipschitz continuous.
\end{remark}


\subsection{Dissipation}
\label{Diss}
Following {\sc Mielke}
\cite{Mielke02,Mielke03b,Mielke04b},  we define the (local)  {\it dissipation distance} $\Delta: (\SL
)^2 \to [0,\infty]$ as
\begin{align*}
\Delta(\Fpo, \Fpj) = \inf\bigg\{
&\int_0^1
R(P(t),\dot P(t))\,
\d t \, :  \,P \in C^1([0,1];\SL),
\ P(i)=\Fpi,\ \text{for} \ i=0,1 \bigg\},
\end{align*}
where the \textit{dissipation potential}
\begin{align*}
R: \SL \times \RR^{d\times d} \rightarrow [0,\infty],
\end{align*}
is convex and positively $1$-homogeneous in the rate,  namely,
\begin{equation*}
R(P,\lambda\dot P)=\lambda R(P,\dot P)\;\;\;\;\text{ for all }\lambda \ge 0,
\end{equation*}
and satisfies the {\it plastic indifference} assumption \cite{Mielke03b}
\begin{equation*}
R(PQ,\dot PQ) = R(P,\dot P) \quad \text{ for all }Q \in \SL.
\end{equation*}
These properties imply that there exists a convex,  positively 1-homogeneous function $\widehat R: \RR^{d\times d} \to [0,\infty]$ such that
\begin{equation*}
R(P,\dot P) = \widehat R(\dot P P^{-1}),
\end{equation*}
see \cite{Mielke03b} or \cite[Section 4.2.1.1]{Mielke-Roubicek} and that $\Delta$ satisfies the {\it triangle inequality} 
\begin{equation*}
\Delta(\Fpo,\Fpoo)\le \Delta(\Fpo,\Fpj)+\Delta(\Fpj,\Fpoo),
\end{equation*}
as well as
\begin{equation*}
\Delta(\Fpo,\Fpj) = \Delta(I,\Fpj \Fpo^{-1})
\end{equation*}
for all $\Fpi \in \SL$, $i=0,1,2$, where $I$ is the identity matrix. 

We assume the function $D: \SL \to [0,\infty]$ defined as
\begin{equation*}
D(\Fp) := \Delta(I,\Fp)
\end{equation*}
to be {\it polyconvex}. Namely, we suppose that there exists a convex function $\haz D: \RR^{d\times d} \times \RR^{d\times d} \to [0,\infty]$ such that
\begin{equation}
\label{polyconvexity:D}
D(\Fp) = \haz D(\Fp, \cof \Fp).
\end{equation}
We refer to \cite[Sec. 4]{Mielke-Mueller} for a discussion about
such polyconvex dissipation potentials. Let us however mention that
this is a delicate point, for a complete characterization of
polyconvex functions $D$ is presently available in the case of 2d
isotropic hardening only.  

Eventually, we define the (global) {\it dissipation distance} between plastic strain states $\Fpo, \Fpj:\Om \to \SL$ as
\begin{equation*}
\D(\Fpo,\Fpj) = \int_\Om D(\Fpj(x)(\Fpo(x))^{-1}) \, \d x
\end{equation*}

and the {\it total dissipation} of a plastic evolution $\yp:[0,T] \to \Q$ from $s$ to $t$ as
\begin{equation*}
\Diss(\nabla \yp ;s,t) = \sup \left\{ \sum_{j=1}^{N} \D(\nabla \yp(t_{i-1}),\nabla \yp(t_i)) : s=t_0 < \dots < t_N =t,\ N \in \NN \right\}.
\end{equation*}

\subsection{Main results}\label{sec:statement} 

 Let  a partition $\Pi=\{0=t_0 < t_1 < \dots < t_{N}=T\}, \ N
\in \NN$ and an initial condition $(\yeo,\ypo) \in \Q$  be given
and let     $(\yei,\ypi) \in \Q$, $i=1,\dots,N$  solve  the incremental minimization problem 
\begin{equation}
\label{incremental-minimization0}
(\yei,\ypi) \in \underset{(\ye,\yp) \in \Q}{\argmin} \Big( \E(t_i,\ye,\yp) + \D(\nabla \ypii,\nabla \yp) \Big).
\end{equation}
 Define  the right-continuous, piecewise-constant 
interpolant 
\begin{align}
\label{approx:def}
&(\ye,\yp)(t) = (\yeii,\ypii) \quad \text{for }t \in [t_{i-1},t_i),
\quad  \ i=1,\dots,N, \nonumber\\
&(\ye,\yp)(T) = (y_{\mathrm{e}N},y_{\mathrm{p}N})
\end{align}
and set
\begin{equation*}
y(t) = \ye(t) \circ \yp(t).
\end{equation*}
We  refer to any such interpolation  $(\ye,\yp): [0,T] \to
\Q$  as   to an  {\it incremental solution}. This solution depends on the choice of minimizers in \eqref{incremental-minimization0} and on the partition $\Pi$.
The following definition is inspired by \cite[Def.~2.12]{DalMaso-Lazzaroni-2010}.

\begin{definition}[Incrementally approximable solutions]\label{def_incr_approx}
We call $(\ye, \yp): [0,T] \to \Q$ an {\it incrementally approximable
quasistatic evolution} if the following conditions are satisfied:
 There  a sequence of partitions $(\Pi_n)_{n \in \NN}$ with fineness
$\max_{i=1,\dots, N(n)} (t^n_i - t^n_{i-1})$ tending to 0 as $n$ goes
to $\infty$  and a corresponding sequence of  incremental
solutions $(\ye^n,\yp^n)_{n\in\NN}\subset\Q$, such that, along not
relabeled subsequences,  
\begin{subequations}
	\label{Helly0}
	\begin{align} 
	&\yp^{n}(t) \wto \yp(t) \text{ in } W^{1,\qp}(\Om), \\
	&\Diss(\nabla \yp^{n};0,t) \to \delta(t), \label{Helly1}\\
	&\Diss(\nabla \yp;s,t) \le \delta(t) - \delta(s) \label{Helly2}
	\end{align}
\end{subequations}
 for some  nondecreasing function $\delta: [0,T] \to
[0,\infty)$ and  for every $s,t \in [0,T]$, and for every $t \in [0,T]$ there exists a $t$-dependent subsequence $n^t_k$ and $\ye(t)$ such that $(\ye^{n_k^t}(t),\yp^{n_k^t}(t))$ converges to $(\ye(t), \yp(t))$ in $\Q$.
Moreover, for all $t \in [0,T]$,
\begin{equation}
\tag{$\mathrm{S_{discr}}$}
\E(t,\ye^n(t),\yp^n(t)) \le \E(t,\hye, \hyp) + \D(\nabla \yp^n(t), \nabla \hyp) \quad \text{for all } (\hye,\hyp) \in \Q,
\end{equation}
and for every $s,t \in \Pi_n, s \le t$,
\begin{equation}
\tag{$\mathrm{E_{discr}}$}
\E(t,\ye^n(t),\yp^n(t) - \E(s,\ye^n(s),\yp^n(s)) + \Diss(\nabla \yp^{n_k};s,t) \le - \int_{s}^{t} \langle \dot \ell(r),(\ye \circ \yp)(r) \rangle \, \d r.
\end{equation}

\end{definition}

\begin{theorem}[Existence of incrementally approximable solutions]
\label{existence:quasistatic}
Let $\Om \subset \RR^d$ be as in Section \emph{\ref{deformations}}. Let $\qe > d, \ \qp \ge d(d-1)$ and define $\Q$ as in Section \emph{\ref{states}}, $\E$ as in Section \emph{\ref{energy}}, and $\D$ as in Section \emph{\ref{Diss}}.
Assume $\We: \RR^{d \times d} \to \RR, \Wp: \SL \to \RR$ and $D: \SL
\to [0,\infty]$ to be polyconvex, see \eqref{polyconvexity} and
\eqref{polyconvexity:D}. Moreover, let $\We$ satisfy the growth
condition \eqref{growth:elastic} and $\Wp$ satisfy the coercivity
bound \eqref{coercive:plastic}. Further, assume that $\ell$ fulfills regularity
assumption \eqref{load:reg}. Let $(\yeo,\ypo) \in \Q$ be initial data
satisfying the {\it semistability} condition at time $0$, namely
\begin{equation*}
\E(0,\yeo,\ypo) \le \E(0,\hye, \hyp) + \D(\nabla \ypo, \nabla \hyp) \quad \text{for all } (\hye,\hyp) \in \Q.
\end{equation*}
Then, there exists an incrementally approximable quasistatic evolution
$(\ye, \yp): [0,T] \to \Q$  with $(\ye(0),\yp(0)) = (y_{\rm e 0},
y_{\rm p 0})$  satisfying the following properties
\begin{align*}
\tag{$\mathrm{S_{semi}}$}
&\E(t,\ye(t),\yp(t)) \le \E(t,\hye, \yp(t)) \quad \text{for all }
\hye \text{ such that } (\hye, \yp(t)) \in \Q ,\\
&
\tag{\textrm{E}}
\E(t,\ye(t),\yp(t)) + \delta(t) \le \E(0,\ye(0),\yp(0)) - \int_0^t \langle \dot \ell(s),(\ye \circ \yp)(s) \rangle \, \d s
\end{align*}
\end{theorem}

Incrementally approximable quasistatic evolutions fulfill the {\it semistability} condition $(\mathrm{S_{semi}})$ with respect to elastic deformations, as well as an {\it energy inequality} \textrm{(E)}. These properties are close to the solution concept discussed in \cite{RoegerSchweizer17} in the context of
viscoplasticity but considerably weaker than the classical notion of {\it energetic solutions} \cite{Mielke-Roubicek}. There, the trajectory is
required to be stable with respect to {\it both} plastic and elastic
deformation and energy equality holds. We refer to \cite[Chapter
3]{Mielke-Roubicek} for a detailed discussion about different solution
concepts for rate-independent systems.

 Despite the weakness of the solution concept, the fact that incrementally approximable
solutions are indeed limits of incremental solutions guarantees that
plasticity actually occurs, whenever necessary. In particular, the
purely elastic evolution $\nabla\yp(t)=I$, which fulfills
$(\mathrm{S_{semi}})$-\textrm{(E)} for compatible initial data, may fail to be incrementally
approximable for loadings exceeding the plastic-activation threshold. 

In
order to give an elementary example of this fact, we present a simplified
argument, by reducing to one space dimension and to a single
material point. In this frame, by choosing energy densities to be 
quadratic and setting all constants to $1$, the incremental problem
\eqref{incremental-minimization0} can be recast in terms of the
deformation strain $f\in \RR$ and the plastic strain $p > 0$ (we neglect
the isochoric constraint, as necessary in one space dimension) as
$$(f_i , p_i) \in \underset{f\in \RR,\, p >0}{\text{argmin}} \left( \frac12|f p^{-1}|^2 +\frac12p^2 -
\ell(t_i) f + |\log p - \log p_{i-1}|\right)\quad \text{for} \ i=1,\dots,N$$
where the initial values $(f_0,p_0)$ with $p_0=1$  and the {loading}
$\ell(t_i) = \lambda t_i $ for $\lambda >0$ are
given. One can prove that $p_i=1$ as long as $|\lambda t_i |\leq 1$. 
In particular, all incrementally approximable
solutions will be such that $p(t)\not =1$ for $|\ell(t)|>1$. In this
case, the purely elastic solution $p(t)=1$ is not incrementally
approximable.



\section{Proofs}\label{proofs}
The proof of Theorem \ref{existence:quasistatic}  is 
detailed along the whole section and consist of several parts. In
Subsection \ref{closedness} we discuss the closure of the state space
$\Q$. In Subsection \ref{existence:incremental} we show the existence
of incrementally approximable solutions. In Subsections
\ref{energy-inequality}-\ref{semistability} the validity of energy
inequality \textrm{(E)} and semistability $(\mathrm{S_{semi}})$ is
 checked   by passing the corresponding approximate properties
of the incremental scheme to the limit. 

\subsection{Closedness of the state space}
\label{closedness}
Consider a sequence $(\ye^n,\yp^n) \in \Q$ converging to $(\ye,\yp)$
in $\Q$ in the sense of Subsection \ref{states}. Then
$\yp^n \wto \yp$ in $W^{1,\qp}(\Om;\RR^d)$  implies the weak
convergence of $\det \nabla \yp^n$ to $\det \nabla \yp$ in
$L^{\qp/d}(\Om;\RR)$. Hence $\det \nabla \yp = 1$ almost
everywhere and $\mathcal{L}^d(\Om) \le \mathcal{L}^d(\yp(\Om))$ by
\cite{ciarlet-necas}. Moreover, strong convergence in $L^\qp(\Om)$
implies $\int_{\Om} \yp (x)\ \d x = 0$. 
The fact that $\yp(\Om) \in \J_{\eps,\delta}$ follows from Lemmas
\ref{closedness:A}-\ref{closedness:B} below. In Lemma
\ref{closedness:A} we show the set convergence \eqref{Hausdorff-conv} by exploiting the fact that the plastic deformations are homeomorphisms, see Subsection \ref{deformations}. In Lemma \ref{closedness:B} we show that $\J_{\eps,\delta}$ is closed under this convergence. 

\begin{lemma}[Hausdorff convergence of intermediate configurations]
	\label{closedness:A}
	Let $\Om \subset \RR^d$ be open and bounded. Let $y, y_n \in C^0(\overline{\Om};\RR^d)$, $n \in \NN,$  be  such that $y_n$ converges to $y$ uniformly on $\overline{\Om}$ and  $y|_{\Om}:\Om \to y(\Om), y_n|_{\Om}:\Om \to y_n(\Om)$  are homeomorphisms for every $n \in \NN$. Then
	\begin{equation*}
	d_{\rm H}(\overline{y_n(\Om)},\overline{y(\Om)}) \to 0, \quad d_{\rm H}(\partial y_n(\Om),\partial y(\Om)) \to 0
	\end{equation*}
	as  $n$ tends to $\infty$.
\end{lemma}
\begin{proof}[ Proof of Lemma \ref{closedness:A}]
	By uniform convergence, we get that 
	$d_{\rm H}(y_n(\overline{\Om}),y(\overline{\Om}))$ and \\$d_{\rm H}(y_n(\partial \Om),y(\partial \Om))$ tend to 0. One is left to show that $y(\overline{\Om}) = \overline{y(\Om)}$ and $y(\partial \Om) = \partial y(\Om)$. 

	Ad $y(\overline{\Om}) = \overline{y(\Om)}$: we observe that,
        by continuity of $y$ and compactness of $\overline{\Om}$,
        the set
        $y(\overline{\Om})$ is closed. Therefore $\overline{y(\Om)}
        \subset y(\overline{\Om})$. In order to check the
        opposite inclusion, let $z
        \in y(\overline{\Om})$ and choose $x \in y^{-1}(z) \subset
        \overline{\Om}$ and $x_n \in \Om$ converging to $x$. Then,
        $y(x_n) \in y(\Om)$ and by the continuity of $y$ up to the
        boundary, $y(x_n)$ converges to $y(x)$, showing that $z = y(x)
        \in \overline{y(\Om)}$. 

	Ad $y(\partial \Om) = \partial y(\Om)$: we use the fact that,
        as $y$ is a homeomorphism, the set $y(\Om)$ is open. Thus, $\partial y(\Om) = y(\overline{\Om})\setminus y(\Om)$. We claim that $y(\overline{\Om})\setminus y(\Om) = y(\partial \Om)$. Indeed, if $z \in y(\Om)$, then there exists an open neighborhood $U$ of $z$ such that $U \subset y(\Om)$. Since $y$ is a homeomorphism, $V = y^{-1}(U)$ is an open neighborhood of $y^{-1}(z)$ such that $V \subset \Om$, implying that $y^{-1}(z) \notin \partial \Om$, i.e. $z \notin y(\partial \Om)$. This shows $y(\overline{\Om})\setminus y(\Om) \supset y(\partial \Om)$. On the other hand, $y(\overline{\Om})\setminus y(\Om) \subset y(\overline{\Om}\setminus \Om) = y(\partial \Om)$. This concludes the proof.
\end{proof}

\begin{lemma}[Closedness of $\J_{\eps,\delta}$ under Hausdorff convergence]
	\label{closedness:B}
	Let $\omega_n \in \J_{\eps,\delta}$ converge to $\omega$ in the sense that
	\begin{equation}
	\label{hausdorff_conv_1}
	d_{\rm H}(\overline{\omega},\overline{\omega}_n) \to 0
	\end{equation}
	and 
	\begin{equation}
	\label{hausdorff_conv_2}
	d_{\rm H}(\partial \omega,\partial \omega_n) \to 0
	\end{equation}
	as  $n$ tends to $\infty$. Then $\omega \in \J_{\eps,\delta}$.
\end{lemma}

\begin{proof}[ Proof of Lemma \ref{closedness:B}]
	Let $x,y \in \omega$ with $|x-y| < \delta$. By convergence \eqref{hausdorff_conv_1}, for every $\nu >0$ there exists $N_\nu \in \NN$ such that for all $n \ge N_\nu$ we have $\omega \subset \overline{\omega} \subset B_\nu(\overline{\omega}_n) = B_\nu({\omega}_n)$. Therefore, we can choose a (not relabeled) subsequence and $x_n,y_n \in \omega_n$ such that $x_n$ and $y_n$ converge to $x$ and $y$, respectively, and $|x_n -y_n| < \delta$ for every $n\in \NN$. We now use the assumption that $\omega_n \in \J_{\eps,\delta}$ and find $\gamma_n \in W^{1,\infty}([0,1];\omega_n)$ such that $\gamma_n(0)=x_n, \gamma_n(1)=y_n$,
	\begin{equation}
	\label{cond:1}
	\ell(\gamma_n) \le \frac{1}{\eps} |x_n-y_n| < \frac{\delta}{\eps},
	\end{equation} 
	and 
	\begin{equation}
	\label{cond:2}
	\text{dist}(\gamma_n(t),\partial \omega_n) \ge \eps \frac{|x_n-\gamma_n(t)||\gamma_n(t)-y_n|}{|x_n-y_n|} \quad \forall t \in [0,1].
	\end{equation}
	Set $L:=\delta/\eps$. From condition \eqref{cond:1} we see that $\sup_{n \in \NN} \ell(\gamma_n) \le L$. Now consider the parametrizations by arclength with constant extension at the endpoint denoted by $\tilde \gamma_n: [0,L] \to \omega_n$. By definition, these satisfy 
	\begin{equation*}
	|\dot{\tilde{\gamma}}_n(s)|=\begin{cases}
	1, &\text{if }s \in [0,\ell(\gamma_n)],\\
	0, &\text{if }s \in (\ell(\gamma_n), L].
	\end{cases}
	\end{equation*}
	We use the Arzel\`a-Ascoli Theorem to extract a (not relabeled) subsequence and find $\tilde \gamma \in W^{1,\infty}([0,L];\overline{\omega})$ such that
	\begin{align}
	&\dot{\tilde{\gamma}}_n \wstar \dot{\tilde{\gamma}} \quad \text{in }L^\infty(0,L),\label{weakstar}\\
	&\tilde \gamma_n \to \tilde \gamma \quad \text{in }C^0([0,L]). \label{strong}
	\end{align}
Define now $\gamma(t):=\tilde{\gamma}(t/L)$. Then $\gamma \in W^{1,\infty}([0,1];\overline{\omega})$ and by weak lower-semicontinuity we get
	\begin{equation*}
	\ell(\gamma) = \int_0^L |\dot{\tilde{\gamma}}(s)| \, \d s \overset{\eqref{weakstar}}{\le} \liminf_{ n \to \infty} \int_0^L |\dot{\tilde{\gamma}}_{ n}(s)| \, \d s = \liminf_{{ n} \to \infty} \ell(\gamma_{ n}) \overset{\eqref{cond:1}}{\le}  \frac{1}{\eps} \lim_{{ n} \to \infty} |x_{ n}-y_{ n}| = \frac{1}{\eps} |x-y|.
	\end{equation*}
	Notice that, as soon as we prove condition
        \eqref{Jones:cond2},  $\gamma([0,1]) \subset \omega$
        follows.
	In order to show \eqref{Jones:cond2}, we fix $s \in [0,L]$. By compactness of the boundary $\partial \omega$ we can choose $z \in \partial \omega$ such that
	\begin{equation}
	\label{dist:quantify}
	\text{dist}(\tilde \gamma(s),\partial \omega) = |\tilde \gamma(s)-z|.
	\end{equation} 
	We  further  choose $z_n \in \partial \omega_n$ such that
	\begin{equation*}
	|z_n-z| \le d_{\rm H}(\partial \omega,\partial \omega_n)
	\end{equation*}
	for every $n \in \NN$. Then, by the triangle inequality
	\begin{align*}
	|\tilde \gamma(s)-z| &\ge |\tilde \gamma_n(s)-z_n| -|\tilde \gamma(s)-\tilde \gamma_n(s)| - |z_n-z| \\
	&\ge \text{dist}(\tilde \gamma_n(s),\partial \omega_n) -\|\tilde \gamma-\tilde \gamma_n\|_{ C^0([0,1])} - d_{\rm H}(\partial \omega,\partial \omega_n).
	\end{align*}
	Using assumption \eqref{hausdorff_conv_2}, condition \eqref{cond:2}, and convergence \eqref{strong},  we  deduce  that
	\begin{equation*}
	|\tilde \gamma(s)-z| \ge \eps \frac{|x_{ n}-\tilde \gamma_{ n}(s)||\tilde \gamma_{ n}(s)-y_{ n}|}{|x_{ n}-y_{ n}|}.
	\end{equation*} 
	 Passage to the limit on the right-hand side  concludes the proof of \eqref{Jones:cond2}.
\end{proof}

\subsection{Existence of incremental solutions}
\label{existence:incremental}
In the following $C>0$ denotes a positive real constant which may
change from line to line, whereas $c>0$ denotes the constant used in
assumptions \eqref{growth:elastic}-\eqref{coercive:plastic}.  For the purpose of readability, we abbreviate  $\|f\|_{L^p(\Om;\RR^d)}$ by $\|f\|_{L^p(\Om)}$ and $\|g\|_{L^p(\Gamma_D;\RR^d)}$ by $\|g\|_{L^p(\Gamma_D)}$.

Let $\Pi = \{0=t_0<t_1<\dots<t_N=T\}$ be a partition of $[0,T]$.
Given $i \in \{1,\dots,N\}$ and $(\yeii,\ypii) \in \Q$ we aim at proving that minimizers of the {\it incremental problem}
\begin{equation}
\label{incremental-minimization}
(\yei,\ypi) \in \underset{(\ye,\yp) \in \Q}{\argmin} \Big( \E(t_i,\ye,\yp) + \D(\nabla \ypii,\nabla \yp) \Big)
\end{equation}
exist.

We follow the Direct Method of the Calculus of Variations: Let $(\ye^n,\yp^n)_{n \in \NN} \subset \Q$ be an infimizing sequence for \eqref{incremental-minimization}. Here, we use that $\Q$ is non-empty, since $(T^{-1}, T) \in \Q$, where $T(x) = x-\bar x$ is a translation and $\bar x$ is the barycenter of $\Om$. As $\D$ is nonnegative, we can assume without loss of generality that $\E(t_i,\ye^n,\yp^n) \le C$. 
We aim at showing the following compactness result:
\begin{equation}
\label{compactness}
\E(t_i,\ye^n,\yp^n) \le C \quad \Longrightarrow \quad (\ye^{n},\yp^{n}) \to (\ye,\yp) \text{ in }\Q
\end{equation}
along a not relabeled subsequence.
Indeed, the energy bound $\E(t_i,\ye^n,\yp^n) \le C$ together with the growth assumption \eqref{growth:elastic} and the coercivity \eqref{coercive:plastic} entails
\begin{equation*}
c\|\nabla \ye^n\|_{L^\qe(\yp^n(\Om))}^\qe + \|\nabla \yp^n\|_{L^\qp(\Om)}^\qp + c\|\ye^n\circ\yp^n - \text{id}\|_{L^1(\Gamma_D)} \le C + \langle \ell(t), \ye^n \circ \yp^n \rangle.
\end{equation*}
By the regularity assumption \eqref{load:reg} and the chain rule
estimate \eqref{chain:est},  we can bound
\begin{align*}
|\langle \ell(t), \ye^n \circ \yp^n \rangle| &\le \|\ell(t)\|_{(W^{1,q}(\Om))^*} \|\ye^n \circ \yp^n \|_{W^{1,q}(\Om)} \\ 
&\le C \|\nabla \ye^n\|_{L^\qe(\yp^n(\Om))} \|\nabla \yp^n\|_{L^\qp(\Om)} \\
&\le \frac{c}{2}\|\nabla \ye^n\|_{L^\qe(\yp^n(\Om))}^\qe + \frac{c}{2} \|\nabla \yp^n\|_{L^\qp(\Om)}^\qp + C
\end{align*}
and conclude that
\begin{equation}
\label{coercive:1}
\|\nabla \ye^n\|_{L^\qe(\yp^n(\Om))}^\qe + \|\nabla \yp^n\|_{L^\qp(\Om)}^\qp + \|\ye^n\circ\yp^n - \text{id}\|_{L^1(\Gamma_D)} \le C.
\end{equation}
Since $\yp^n$ has zero mean, the Poincar\'e-Wirtinger inequality implies that $\yp^n$ is bounded in $W^{1,\qp}(\Om;\RR^d)$. On the other hand  $y^n := \ye^n \circ \yp^n$  is subject to the elastic Dirichlet boundary condition on $\Gamma_D$ and we have the following result.

\begin{lemma}[Generalized Poincar\'e inequality]
	\label{Poincare}
	Let $\Om \subset \RR^d$ be as in Section \emph{\ref{deformations}} and $q \ge 1$. Then, there exists a constant $C_{\text{Poincar\'e}} > 0$ such that 
	\begin{equation*}
	\|y\|_{W^{1,q}(\Om)} \le C_{\text{Poincar\'e}}\Big( \|\nabla y\|_{L^q(\Om)} + \|y - \mathrm{id}\|_{L^1(\Gamma_D)} \Big)
	\end{equation*}
	for every $y \in W^{1,q}(\Om;\RR^d)$.
\end{lemma}
\begin{proof}[ Proof of Lemma \ref{Poincare}]
	We argue by contradiction. Let the sequence $(y_k)_{k \in \NN} \subset W^{1,q}(\Om;\RR^d)$ be such that  
	\begin{equation}\label{contradict1}
	\|\nabla y_k\|_{L^q(\Om)} + \|y_k-\text{id}\|_{L^1(\Gamma_D)} < \frac{1}{k} \|y_k\|_{L^q(\Om)}.
	\end{equation}
	We claim that $\|y_k\|_{L^q(\Om)} \to \infty$. If this would not be the case, we would have $\|y_k\|_{W^{1,q}(\Om)} \le C$ and can pick a (not relabeled) subsequence $y_k$ converging to $y$ weakly in $W^{1,q}(\Om)$. By the trace theorem, the traces would also converge strongly in $L^1(\partial \Om)$. Moreover, by \eqref{contradict1},
	\begin{equation*}
	\|\nabla y\|_{L^q(\Om)} + \|y-\text{id}\|_{L^1(\Gamma_D)} \le \liminf_{k\to \infty} \Big(\|\nabla y_k\|_{L^q(\Om)} + \|y_k-\text{id}\|_{L^1(\Gamma_D)}\Big) = 0.
	\end{equation*}
	This would imply $\nabla y = 0$  in  $\Om$ and $y = \mathrm{id}$ on $\Gamma_D$. A contradiction.  Hence  $\|y_k\|_{L^q(\Om)} \to \infty$.\\
	We now rescale the sequence by setting
	\begin{equation*}
	w_k := \frac{y_k}{\|y_k\|_{L^q(\Om)}}
	\end{equation*} 
	and note that
	\begin{equation*}
	\|w_k\|_{L^q(\Om)} = 1 \qquad  \text{and}  \qquad \|\nabla w_k\|_{L^q(\Om)} + \|w_k- \lambda_k \text{id} \|_{L^1(\Gamma_D)} < \frac{1}{k}
	\end{equation*}
	where  $\lambda_k = \|y_k\|_{L^q(\Om)}^{-1}$ tends to 0.
	Then, we  choose  a (not relabeled) subsequence $w_k$ converging to  some $w$ weakly in $W^{1,q}(\Om)$, strongly in $L^q(\Om)$, and such that the traces converge strongly in $L^1(\partial \Om)$. This leads to $\nabla w = 0$ in $\Om$ and $w = 0$ on $\Gamma_D$. Since $\Om$ is connected, this forces $w=0$ in $\Om$ and contradicts the fact that $\|w\|_{L^q(\Om)}=\lim_{k\to \infty}\|w_k\|_{L^q(\Om)} = 1$.
\end{proof}

We use Lemma \ref{Poincare} in combination with the chain rule \eqref{chain} and H\"older's inequality to estimate
\begin{align}
\label{Poincare_all}
\|y^n\|_{W^{1,q}(\Om)}^q &\le C \left(\|\nabla \ye^n\|_{L^\qe(\yp^n(\Om))}^\qe + \|\nabla \yp^n\|_{L^\qp(\Om)}^\qp + \|\ye^n\circ\yp^n - \text{id}\|_{L^1(\Gamma_D)}^q\right) 
\overset{\eqref{coercive:1}}{\le} C.
\end{align}
We further remark that $W^{1,q}(\Om;\RR^d)$ embeds into $L^{q^*}(\Om;\RR^d)$ with $q^* > \qe$. This leads to
\begin{equation*}
\|\ye^n\|_{L^\qe(\yp^n(\Om))} = \|y^n\|_{L^\qe(\Om)} \le C\|y^n\|_{W^{1,q}(\Om)} \le C.
\end{equation*}
Altogether, we conclude that
\begin{equation*}
\|y^n\|_{W^{1,q}(\Om)} + \|\ye^n\|_{W^{1,\qe}(\yp^n(\Om))} + \|\yp^n\|_{W^{1,\qp}(\Om)} \le C.
\end{equation*}
This bound implies that there exists a (not relabeled) subsequence such that  $(\ye^n,\yp^n)$ converges to $(\ye,\yp)$ in $\Q$  which concludes the proof of \eqref{compactness}.

In Section \ref{closedness} we have seen that  $\Q$ is closed under this convergence, consequently $(\ye,\yp) \in \Q$. Furthermore, by the continuity of the trace operator, we have $y^n \to y$ strongly in $L^{q^\#}(\partial \Om)$, where $q^\# > (d-1)^2 \ge 1$. This yields
\begin{equation}
\label{lsc1}
\int_{\Gamma_D} |y(x) - x| \, \d \mathcal H^{d-1}(x) = \lim_{n\to \infty} \int_{\Gamma_D} |y^n(x) - x| \, \d \mathcal H^{d-1}(x). 
\end{equation}
 By the weak continuity of the loading term, we have that $\langle
\ell(t_i),y^n \rangle$ converges to $\langle \ell(t_i),y \rangle$. The weak continuity of the minors entails that $\cof \nabla \yp^n \wto \cof \nabla \yp$ in $L^{\qp/(d-1)}(\Om;\RR^d)$. In combination with polyconvexity \eqref{polyconvex:2}, we deduce
\begin{equation}
\label{lsc2}
\int_{\Om} W_{\text{p}}(\nabla \yp(x)) \ \d x\le \liminf_{n \to \infty} \int_{\Om} W_{\text{p}}(\nabla \yp^n(x)) \ \d x.
\end{equation}
For every fixed $K \subset \subset \yp(\Om)$, again by weak continuity
of the minors (recall that $\qe > d$) and polyconvexity \eqref{polyconvex:1}, we have 
\begin{equation}
\label{lsc3}
\int_{K} W_{\text{el}}(\nabla \ye(\xi)) \ \d \xi \le \liminf_{n \to \infty} \int_{\yp^n(\Om)} W_{\text{el}}(\nabla \ye^n(\xi)) \ \d \xi.
\end{equation}
Letting $K$ tend to $\yp(\Om)$ in \eqref{lsc3}, together with \eqref{lsc1} and \eqref{lsc2} we have shown lower semi-continuity of the energy, namely
\begin{equation}
\label{lower-semicontinuity:energy}
\E(t_i,\ye,\yp) \le \liminf_{n\to \infty} \E(t_i,\ye^n,\yp^n).
\end{equation}
In a similar way, using polyconvexity of $D$ \eqref{polyconvexity:D},
we get
\begin{align*}
  \D(\nabla \ypii,\nabla \yp) &= \int_\Omega D (\nabla \yp(\nabla
  \ypii)^{-1})\, \d x \\
&= \int_\Omega \haz D (\nabla \yp(\nabla
  \ypii)^{-1}, \cof (\nabla \yp(\nabla
  \ypii)^{-1}))\, \d x\\
&=\int_\Omega \haz D (\nabla \yp(\nabla
  \ypii)^{-1}, \cof (\nabla \yp)\,\cof (\nabla
  \ypii)^{-1})\, \d x \\
&
\leq \liminf_{n\to \infty}
\int_\Omega \haz D (\nabla \yp^n(\nabla
  \ypii)^{-1}, \cof (\nabla \yp^n)\cof (\nabla
  \ypii)^{-1})\, \d x \\
&= \liminf_{n\to \infty}\D(\nabla \ypii,\nabla \yp^n),
\end{align*}
where we also used the fact that 
\begin{align*}
  &\cof (\nabla \yp^n)\cof (\nabla
  \ypii)^{-1} = \cof (\nabla \yp^n) \nabla
  \ypii^{T} \\
&\wto \cof (\nabla \yp) \nabla
  \ypii^{T}  =  \cof (\nabla \yp)\cof (\nabla
  \ypii)^{-1} 
\end{align*}
in $L^{q_{\rm p}/d}(\Omega;\RR^{d\times d})$.
This shows that $(\ye,\yp)$ is a minimizer of \eqref{incremental-minimization}.

The discrete stability condition $\mathrm{(S_{discr})}$ can be deduced easily by testing \eqref{incremental-minimization} with a competitor $(\hye,\hyp) \in \Q$ and using the triangle inequality for $\D$. The discrete energy inequality $\mathrm{(E_{discr})}$ is shown below in \eqref{energy-inequality:n}.

\subsection{Energy inequality}
\label{energy-inequality}
Take a sequence of partitions $\Pi_n=\{0=t_0^n < t_1^n < \dots <
t_{N(n)}^n=T\}, \ n \in \NN,$ with fineness
$\max_{i=1,\dots, N(n)} (t^n_i - t^n_{i-1})$ tending to 0 as $n$ goes to $\infty$. For fixed $n$ we iteratively choose $(\yei^n,\ypi^n) \in \Q$, $i=1,\dots,N(n)$, solving the incremental minimization problem \eqref{incremental-minimization},   consider the right-continuous, piecewise constant approximation as in \eqref{approx:def}, 
and set
\begin{equation*}
y^n(t) = \ye^n(t) \circ \yp^n(t).
\end{equation*}
Testing \eqref{incremental-minimization} against $(\yeii^n,\ypii^n)$, we get
\begin{equation*}
\E(t_i^n,\yei^n,\ypi^n) - \E(t_{i-1}^n,\yeii^n,\ypii^n) + \D(\nabla \ypii^n,\nabla \ypi^n) \le - \int_{t_{i-1}^n}^{t_i^n} \partial_t \E(s,\yeii^n,\ypii^n) \, \d s.
\end{equation*}

Summing over $i$, we arrive at
\begin{equation}
\label{energy-inequality:n}
\E(t_k^n,\yek,\ypk) - \E(t_j^n,\yej,\ypj) + \Diss(\nabla \yp^n;t_j^n,t_k^n) \le - \int_{t_j^n}^{t_k^n} \partial_t \E(s,y^n(s)) \,\d s
\end{equation}
for every $0 \le j \le k \le N(n)$.  We estimate the right-hand side by
\begin{equation*}
|\partial_t \E(s,y^n(s))| \le \zeta(s)\|y^n(s)\|_{W^{1,q}(\Om)},
\end{equation*}
where $\zeta(s):=\|\dot \ell(s)\|_{W^{1,q}(\Om)^*}$, by assumption \eqref{load:reg}, is integrable. By estimate \eqref{Poincare_all} and assumptions \eqref{growth:elastic}-\eqref{coercive:plastic}, we have
\begin{equation*}
\|y^n(s)\|_{W^{1,q}(\Om)} \le C(1+\E(s,\ye^n(s),\yp^n(s))).
\end{equation*}
Therefore, altogether

\begin{equation*}
\E(t,\ye^n(t),\yp^n(t)) - \E(s,\ye^n(s),\yp^n(s)) + \Diss(\nabla \yp^n;s,t) \le C \int_{s}^{t} \zeta(r)(1+ \E(r,y^n(r))) \,\d r
\end{equation*}
for every $s,t\in \Pi_n, s \le t$.  By virtue of Gronwall's inequality, using the integrability of $\rho$, we find
\begin{equation*}
\sup_{t \in \Pi_n} \E(t,\ye^n(t),\yp^n(t)) \le C.
\end{equation*}
Since $\E$ is absolutely continuous in time and the approximate solution $(\ye^n,\yp^n)$ is piecewise constant, we deduce
\begin{equation}
\label{bound}
\sup_{t \in [0,T]} \E(t,\ye^n(t),\yp^n(t)) + \Diss(\nabla \yp^n;0,T)\le C.
\end{equation}
We now prepare an intermediate result.
\begin{lemma}[Lower-semicontinuity of $\D$ in both arguments]
	\label{lsc:D}
	Let $\yp^n \wto \yp$ and $\ypo^n \wto \ypo$ in $W^{1,\qp}(\Om)$ with $\qp > d(d-1)$ such that $\det \nabla \ypo^n = 1$ a.e. and $|\Om| \le |\ypo^n(\Om)|$ for every $n \in \NN$. Then,
	\begin{equation*}
	\D(\nabla \ypo,\nabla \yp) \le \liminf_{n \to \infty} \ \D(\nabla \ypo^n,\nabla \yp^n).
	\end{equation*}
\end{lemma}
\begin{proof}[ Proof of Lemma \ref{lsc:D}]
	We  rely on  the assumption $\qp \ge d(d-1)$ and define 
	\begin{equation*}
	v^n = \yp^n \circ (\ypo^n)^{-1},
	\end{equation*} 
	where the global inverse $(\ypo^n)^{-1}$ is bounded in
        $W^{1,\qp/(d-1)}(\ypo^n(\Om))$, since we have that $(\nabla \ypo^n)^{-1} = (\cof \nabla\ypo^n)^T$.
	We rewrite
	\begin{equation*}
	\D(\nabla \ypo^n,\nabla \yp^n) = \int_{\ypo^n(\Om)} D(\nabla v^n(\xi)) \, \d \xi
	\end{equation*}
	and estimate
	\begin{subequations}
	\label{est}
	\begin{align}
	\label{est:1}
	\|v^n \|_{L^{\qp/d}(\ypo^n(\Om))} &= \|\yp^n\|_{L^{\qp/d}(\Om)}
        \le |\Om|^{(d-1)/\qp} \|\yp^n\|_{L^{\qp}(\Om)} \le C,\\
	\|\nabla v^n \|_{L^{\qp/d}(\ypo^n(\Om))} &= \|\nabla\yp^n
        (\nabla\ypo^n)^{-1}\|_{L^{\qp/d}(\Om)} \nonumber\\
 &\le \|\nabla\yp^n\|_{L^\qp(\Om)} \|\cof \nabla\ypo^n\|_{L^{\qp/(d-1)}(\Om)} \le C. 	\label{est:2}
	\end{align}
	\end{subequations}
	Let $K$ be a compact subset of $\ypo(\Om)$. Since $\ypo^n \to \ypo$ uniformly in $\overline \Om$, there exists $n_K \in \NN$ such that for all $n \ge n_K$, we have $K \subset \ypo^n(\Om)$. Using estimates \eqref{est},  we choose a (not relabeled) subsequence such that
	\begin{equation*}
	v^n \wto v \quad \text{in } W^{1,\qp/d}(K),
	\end{equation*}
	where $v = \yp \circ \ypo^{-1}$ on $K$. As $\qp/d > d-1$, we conclude, by using the polyconvexity \eqref{polyconvexity:D} and the weak continuity of the minors of $\nabla v$, that
	\begin{equation*}
	\int_{K} D(\nabla v(\xi)) \, \d \xi \le \liminf_{n \to \infty} \int_{K} D(\nabla v^n(\xi)) \, \d \xi \le \liminf_{n \to \infty} \int_{\ypo^n(\Om)} D(\nabla v^n(\xi)) \, \d \xi.
	\end{equation*}
	Now it suffices to consider an increasing sequence of compact subsets exhausting $\ypo(\Om)$. By further extracting a diagonal sequence, we get that $\yp = v \circ \ypo$ on $\Om$ and the statement follows.
\end{proof}

We proceed with the proof of the energy inequality by noting that \eqref{bound} together with Lemma \ref{lsc:D} allows us to use Helly's Selection Principle \cite[Thm. B.5.13]{Mielke-Roubicek}. Namely, there exists a subsequence $(n_k)_{k\in \NN}$, a function $\yp: [0,T] \to W^{1,\qp}(\Om)$, and a nondecreasing function $\delta: [0,T] \to [0,\infty)$ such that

\begin{subequations}
	\label{Helly}
	\begin{align} 
	&\yp^{n_k}(t) \wto \yp(t) \text{ in } W^{1,\qp}(\Om), \\
	&\Diss(\nabla \yp^{n_k};0,t) \to \delta(t), \label{Helly:3}\\
	& \Diss(\nabla \yp;s,t) \le \delta(t) - \delta(s) \label{Helly:4}
	\end{align}
\end{subequations}
for every $s,t \in [0,T]$.  By defining $\theta_n(s) := -\langle \dot
\ell(s), y^n(s) \rangle$ and observing that $\theta_n$ is
equiintegrable, we can use the Dunford-Pettis Theorem (see \cite{DunPet} or \cite[Theorem B.3.8]{Mielke-Roubicek}) to extract a further (not relabeled) subsequence satisfying
\begin{equation}
\label{Dunford-Pettis}
\theta_{n_k} \wto \theta \quad \text{in } L^1(0,T).
\end{equation}
Fix now some $t\in [0,T]$ and define 
\begin{equation*}
\tau^n := \min\{\tau \in \Pi_n: \tau \ge t\}
\end{equation*}
 such that $\tau^n \ge t, \tau^n \to t$. We can directly pass to the $\liminf$ in the dissipation
 
\begin{equation*}
\delta(t) \overset{\eqref{Helly:3}}{=} \lim_{k \to \infty} \Diss(\nabla \yp^{n_k};0,t) \le \liminf_{k \to \infty} \Diss(\nabla \yp^{n_k};0,\tau^{n_k}).
\end{equation*}

Moreover, by the energy bound \eqref{bound}, we can follow the
argument leading to \eqref{compactness} and
choose a $t$-dependent subsequence $(N_k^t)_{k \in \NN}$ of $(n_k)_{k
  \in \NN}$ such that $(\ye^{N_k^t}(t),\yp^{N_k^t}(t))$ converges to $(\ye(t),
\yp(t))$ in $\Q$ and $y^n(t) \wto \ye(t)\circ \yp(t)=:y(t)$ in
$W^{1,q}(\Om;\RR^d)$. Additionally, by extracting a further subsequence, we guarantee that 
\begin{equation*}
\theta^{N_k^t}(t) \to \limsup_{k \to \infty} \theta^{n_k}(t) := \theta_{\rm sup}(t).
\end{equation*}

Since $y^{N_k^t}(t) \wto y(t)$ in $W^{1,q}(\Om)$, it easily follows that $$\theta_{\rm sup}(t) = \lim_{k\to \infty} \theta^{N_k^t}(t) = \lim_{k\to \infty} \langle \dot \ell(t), y^{N_k^t}(t) \rangle = \langle \dot \ell(t), y(t) \rangle.$$
Furthermore,
\begin{equation*}
\E(t,\ye(t),\yp(t)) \le \liminf_{n \to \infty} \E(\tau^{N_k^t},\ye^{N_k^t}(t),\yp^{N_k^t}(t)),
\end{equation*} 
see the discussion in Section \ref{existence:incremental} leading to \eqref{lower-semicontinuity:energy} and notice that $(\ye^{n},\yp^{n})(t) = (\ye^{n}, \yp^{n})(\tau^{n})$. 
At this point, we can pass to the $\limsup$ on the right-hand side of inequality \eqref{energy-inequality:n}, using convergence \eqref{Dunford-Pettis} and $\theta \le \theta_{\rm sup}$. As the energy is continuous in $t$, we conclude that

\begin{equation*}
\E(t,{\ye}(t),{\yp}(t)) - \E(0,\yeo,\ypo) + \delta(t) \le - \int_{0}^{t} \langle \dot \ell(s),y(s) \rangle \, \d s
\end{equation*}

as desired.

\subsection{Semistability}
\label{semistability}
 In this section, we prove the semistability condition $\mathrm{(S_{semi})}$. 
Fix $t \in [0,T]$ and define $\tau^n$ as  above. Let
$(\ye^n(t),\yp^n(t))$ be the right-continuous approximation defined in
\eqref{approx:def} and note that $(\ye^n,\yp^n)(t) =
(\ye^n,\yp^n)(\tau^n)$. By testing the minimum in
\eqref{incremental-minimization} at time  $\tau^n$  against competitors
$(\hyen,\yp^n(t))$ having the same plastic component,  we get the following \textit{discrete} semistability:
\begin{equation}
\label{semistability:n}
\E(\tau_n,\ye^n(t),\yp^n(t)) \le \E(\tau_n,\hyen, \yp^n(t))
\end{equation}
for every $\hyen$ satisfying $(\hyen,\yp^n(t))\in \Q$. By following
the discussion of Subsections \ref{existence:incremental}-\ref{energy-inequality}, we can choose a (not relabeled) subsequence such that $(\ye^n(t),\yp^n(t)) \to (\ye(t),\yp(t))$ in $\Q$. Note that this subsequence may be $t$-dependent as in Section \ref{energy-inequality}.
 We aim at showing the corresponding {\it limit} semistability:
\begin{equation}
\label{semistability:again}
\E(t,\ye(t),\yp(t)) \le \E(t,\hye, \yp(t))
\end{equation}
for every $\hye$ satisfying $(\hye,\yp(t))\in \Q$. This is done by passing to the limit in \eqref{semistability:n} with a suitable recovery sequence in the spirit of \cite{MRS08}. 

\begin{lemma}[Existence of recovery sequences]\label{mrs}
	 Let $(\ye^n,\yp^n)_{n \in \NN} \subset \Q$ converge to $(\ye,\yp)$ in $\Q$. Then, for every $\hye$ with $(\hye,\yp) \in \Q$ there exists a sequence $\hyen$ with $(\hyen,\yp^n) \in \Q$ satisfying
	\begin{equation}
	\label{limsup:est}
	\limsup_{n \to \infty} \, \Big( \E(\tau_n, \hyen, \yp^n) - \E(\tau_n,\ye^n,\yp^n) \, \Big) \le \E(t,\hye,\yp) - \E(t,\ye,\yp).
	\end{equation}
\end{lemma}

\begin{proof}[Proof of Lemma 3.5]
    Let $\hye$ be such that $(\hye,\yp) \in \Q$. As the
    energy $\mathcal E$ is  absolutely continuous  with respect to time \eqref{load:reg}, relation \eqref{limsup:est} follows as soon as we find $\hyen$
    with $(\hyen,\yp^n) \in \Q$ satisfying
    \begin{equation}
    \label{limsup:estnew}
    \limsup_{n \to \infty} \, \Big( \E(t,\hyen, \yp^n) - \E(t,\ye^n,\yp^n) \, \Big) \le \E(t,\hye,\yp) -     \E(t,\ye,\yp)
    \end{equation}
    where now time is fixed.

    In order to check for \eqref{limsup:estnew}, due to the cancellation of the kinematic hardening energy $\Wp$ and the weak lower-semicontinuity \eqref{lower-semicontinuity:energy}, it suffices to show that 	there exists a sequence $(\hyen,\yp^n) \in \Q$ such that
	\begin{equation}\label{elastic:recovery}
	\limsup_{n \to \infty} \int_{\yp^n(\Om)} \We(\nabla \hyen(\xi))\ \d \xi \le \int_{\yp(\Om)} \We(\nabla \hye(\xi))\ \d \xi
	\end{equation}
	and 
	\begin{equation}
	\label{constraints:recovery}
	\lim_{n\to \infty} \int_{\Gamma_D} |\hyen(\yp^n(x)) - x| \, \d \mathcal H^{d-1}(x) = \int_{\Gamma_D} |\hye(\yp(x)) - x| \, \d \mathcal H^{d-1}(x). 
	\end{equation}
	We may assume without loss of generality that
	\begin{equation*}
	\int_{\yp(\Om)} \We(\nabla \hye(\xi))\ \d \xi  < \infty.
	\end{equation*}
	 In particular, $\hye \in W^{1,\qe}(\yp(\Om);\RR^d)$ by the growth assumption \eqref{growth:elastic}. To define the recovery sequence $\hyen$ we use the fact that $\yp(\Om) \in \J_{\eps,\delta}$. By the extension property of $(\eps,\delta)$-domains, there exists a bounded linear operator
	\begin{equation}\label{def:E}
	E: W^{1,\qe}(\yp(\Om);\RR^d) \to W^{1,\qe}(\RR^d;\RR^d)
	\end{equation}
	such that $Eu = u$ in $\yp(\Om)$, and a constant $C$ solely depending on $\eps,\delta,\qe$, and $d$ such that
	\begin{equation}
	\label{extention:bounded}
	\|E u \|_{W^{1,\qe}(\RR^d)} \le C \|u \|_{W^{1,\qe}(\yp(\Om))}
	\end{equation}
	for every $u \in W^{1,\qe}(\yp(\Om);\RR^d)$ \cite[Theorem 1]{Jones}.\\
	Set now
	\begin{equation*}
	\hyen := E \hye \big|_{\yp^n(\Om)}
	\end{equation*}
	and note that this test is admissible, namely $(\hyen,\yp^n) \in \Q$. Then, we split
	\begin{equation*}
	\int_{\yp^n(\Om)} \We(\nabla \hyen(\xi))\, \d \xi = \int_{\yp^n(\Om) \cap \yp(\Om)} \We(\nabla \hye(\xi))\, \d \xi + \int_{\yp^n(\Om) \setminus \yp(\Om)} \We(\nabla (E \hye)(\xi))\, \d \xi.
	\end{equation*}
	We use the growth condition \eqref{growth:elastic} to control
	\begin{equation}
	\label{growth:est}
	\left| \int_{\yp^n(\Om) \setminus \yp(\Om)} \We(\nabla (E \hye)(\xi))\, \d \xi \right| \le \frac{1}{c} \int_{\yp^n(\Om) \setminus \yp(\Om)} \left( 1 + |\nabla (E\hye)(\xi)|^\qe \right) \ \d \xi 
	\end{equation} 
	and use the convergence $\mathcal L^d(\yp^n(\Om) \setminus \yp(\Om)) \to 0$ as well as $E\hye \in W^{1,\qe}(\RR^d;\RR^d)$, by bound \eqref{extention:bounded}, to deduce the convergence to  0  of the right-hand side of \eqref{growth:est}. 
	Since $\yp^n \to \yp$ uniformly  in  $\overline \Om$, we have $\mathcal L^d(\yp^n(\Om) \triangle \yp(\Om)) \to 0$ as $n\to \infty$, where $\triangle$ denotes the symmetric difference. Thus,
	\begin{equation*}
	\lim_{n \to \infty} \int_{\yp^n(\Om) \cap \yp(\Om)} \We(\nabla \hye(\xi))\, \d \xi = \int_{\yp(\Om)} \We(\nabla \hye(\xi))\, \d \xi.
	\end{equation*}
 This proves  inequality \eqref{elastic:recovery} and  we
 are left with  checking  the convergence \eqref{constraints:recovery}. 
	Observe that by the chain rule \eqref{chain}-\eqref{chain:est} and the boundedness of the extension \eqref{extention:bounded},  we have
	\begin{subequations}
		\begin{align*}
	&	\|\hyen \circ \yp^n\|_{L^q(\Om)} =
                \|E\hye\|_{L^q(\yp^n(\Om))} \le
                |\yp^n(\Om)|^{1/\qp}\|E\hye\|_{L^\qe(\RR^d)} \le C,\\
    &
		\|\nabla(\hyen \circ \yp^n)\|_{L^q(\Om)} \le \|\nabla E\hye\|_{L^\qe(\yp^n(\Om))} \|\nabla \yp^n\|_{L^\qp(\Om)} \le C.
		\end{align*}
	\end{subequations}
	The latter shows that $\hyen \circ \yp^n$
    is bounded in $W^{1,q}(\Om;\RR^d)$. Moreover, we know that
    $\yp^n$ converges uniformly to $\yp$ on $\overline{\Om}$. Hence
    we can choose a (not relabeled) subsequence such that $E\hye
    \circ \yp^n \wto \hye \circ \yp$ in
    $W^{1,q}(\Om;\RR^d)$. This implies that $E\hye \circ \yp^n \to \hye \circ \yp$ in $L^{q^\#}(\partial \Om)$ with $q^\# > (d-1)^2 \ge 1$, so that the convergence
	\begin{equation*}
	\int_{\Gamma_D} |E\hye(\yp^n(x))-x| \ \d \mathcal H^{d-1}(x) \to \int_{\Gamma_D} |\hye(\yp(x))-x| \ \d \mathcal H^{d-1}(x)
	\end{equation*}
	holds. 
	\end{proof}

By using Lemma \ref{mrs}, starting from the discrete
semistability \eqref{semistability:n},  we readily check its
time-continuous counterpart \eqref{semistability:again} for all
times. This concludes the proof of Theorem
\ref{existence:quasistatic}.

\begin{remark}
We wish to point out that, to our best knowledge, necessary and (nontrivial) sufficient  conditions allowing for an extension of a deformation mapping  (as in \eqref{def:E}) but respecting  additionally the orientation-preservation of the map are not known in general. We refer to \cite{BBMK} for some discussion on this issue.  
\end{remark}

\section*{Acknowledgments}
This work has been funded by the Austrian Science Fund (FWF) and Czech
Science Foundation (GA\v CR) through the joint international project
FWF-GA\v CR I\,4052, 19-29646L,  and the scientific and technological cooperation project OeAD-M\v SMT CZ 17/2016,  8J19AT013.
  Support  by the Austrian Science Fund (FWF) projects F\,65
 and W\,1245  is also acknowledged.

\end{document}